\documentclass[oneside,english]{amsart}
\usepackage[]{epsf,epsfig,amsmath,amssymb,amsfonts,latexsym}
\title{Gibbs and equilibrium measures for some families of subshifts}
\author{Tom Meyerovitch}

 \theoremstyle{plain}
\newtheorem{thm}{Theorem}[section]
\newtheorem{lem}{Lemma}[section]
\newtheorem{prop}[thm]{Proposition}
\newtheorem{cor}[thm]{Corollary}
\newtheorem*{thm*}{Theorem}

\newcommand{\ZZ}{\mathbb{Z}}

\newcommand{\hX}{\widehat{X}}
\newcommand{\RR}{\mathbb{R}}
\newcommand{\NN}{\mathbb{N}}
\newcommand{\ZD}{{\mathbb{Z}^d}}

\newcommand{\htop}{h_{\mathit{top}}}
\newcommand{\T}{\mathfrak{T}}

\newcommand{\Homeo}{\mathit{Homeo}}
\newcommand{\SVD}{\mathit{SV}_d}

\begin{document}
\begin{abstract}
%We prove a generalization of the Lanford-Ruelle theorem which is
%valid for general subshifts (non necessarily of finite type), and
%study the interplay between Gibbs measures and equilibrium measures
%for three types of subshifts: The Dyck shift, Kalikow
%type-subshifts, and $\beta$-shifts.

 For SFTs, any equilibrium
measure is Gibbs, as long a $f$ has $d$-summable variation. This is
a theorem of Lanford and Ruelle. Conversely, a theorem of
Dobru{\v{s}}in states that for strongly-irreducible subshifts,
shift-invariant Gibbs-measures are equilibrium measures.

%We investigate the relationship of Gibbs and equilibrium measures
%for subshifts which are not of finite type.
Here we prove a
generalization of the Lanford-Ruelle theorem: for all subshifts, any
equilibrium measure for a function with $d$-summable variation is
``topologically  Gibbs''. This is a relaxed notion which coincides
with the usual notion of a Gibbs measure for SFTs.

 In the second part of the paper, we study Gibbs
and equilibrium measures for some interesting
 families of subshifts: $\beta$-shifts, Dyck-shifts and Kalikow-type
shifts (defined below). In all of these cases, a Lanford-Ruelle type
theorem holds. For each of these families we provide a specific
proof of the result.

\end{abstract}

\maketitle

\section{Introduction}

In the starting point of our study are a couple of related theorems
originating in mathematical physics: One is a theorem of Lanford and
Ruelle and the other is a theorem of Dobru{\v{s}}in. Phrased in the
terminology of symbolic dynamics, these theorems deal with
equilibrium and Gibbs measures for subshifts of finite type. Since
these theorems aim to reflect a physical theory, it is of interest
to explore the ``robustness'' of the phenomena which these theorems
describe, by relaxing the assumptions of the mathematical model. Our
effort is to explore the validity of the conclusions of these
theorems for subshifts which are not of finite type.

Section \ref{sec:prem_and_defs} is an overview classical definition
and results in this field, along with some new definitions such as
the ``topological Gibbs relation'', which are needed for a concise
formulation of some of our results.

In section \ref{sec:dobru} we state and explain the classical
Dobru{\v{s}}in  and  Lanford-Ruelle theorems.  We then formulate and
prove generalized versions of the Lanford-Ruelle theorem, and prove
a simple example for the breakdown of the ``obvious''
generalization.

These first two sections deal with a very general setting which
includes both one dimensional and multidimensional subshifts. The
last $3$ sections specialize with some specific families of
non-sofic $1$-dimensional shifts, each of which is of special
interest: Section \ref{sec:kalikow} investigates a  family of
subshifts which we name ``Kalikow-type subshifts'', which have a
natural definition in terms of a skew-product. For these shifts
there is a total breakdown of rigidity for Gibbs measures, yet the
conclusion of the Lanford-Ruelle theorem holds for measures of
maximal entropy. In section \ref{sec:beta} $\beta$-shifts are shown
to have a unique tail-invariant measure. Section \ref{sec:dyck}
contains a proof of a restricted form of Lanford-Ruelle theorem for
the Dyck shift, which is a supplement to the results of
\cite{meyerovitch_dyck}.

\textbf{Acknowledgments:} This work is a part of the author's Ph.D
thesis, written in Tel-Aviv University under the supervision of
Professor Jon Aaronson. The author acknowledges the support of
support of the Crown Family Foundation Doctoral Fellowships ,USA.

\section{\label{sec:prem_and_defs}Preliminaries and definitions}

\subsection{The Gibbs relation} Let $T:X \to X$ be a homeomorphism of
a compact metric space. The \emph{Gibbs} relation of $(X,T)$ (also
called \emph{homoclinic} relation, or \emph{double-tail} relation
\cite{aaro-nakada-2007,petersen_schmidt97}) is defined as the pairs
of points in $X$ which have asymptotically converging orbits:
$$\T(X,T) := \{ (x,y) \in X \times X :~ \lim_{|n|\to \infty}d(T^nx,T^ny)=0\}$$
We abbreviate  this either by  $\T=\T(X,T)$  or by $\T_X=T(X,T)$,
according to the context.

$\T$ is an equivalence relation. Denote the $\T$-equivalence class
of $x \in X$ by $\T(x)=\{y \in X : ~ (x,y)\in \T\}$.

\begin{lem}\label{lem:gibbs_ctble_expansinve}
Whenever $T$ is expansive, $\T(x)$ is at most countable for all $x
\in X$
\end{lem}
\begin{proof}
Let $\epsilon>0$ be an expansive constant for $T$. Thus, $x\ne y$
implies $d(T^n x,T^n y) > \epsilon$ for some $n \in \ZZ$. Fix $x \in
X$. We have:
$$\T(x) \subset  \bigcup_{n \in \NN}\{y \in X :~  d(T^kx,T^ky)<\frac{\epsilon}{2} ~\forall  |k|>n\}$$
%and so by expansiveness:
%$$\T(x) \setminus \{x\} \subset \bigcup_{n \in \NN}\{ y \in X:~ \max_{|k| < n}d(T^kx,T^ky)>\epsilon\}$$

Since $T$ is expansive, it follows that for each $n \in \NN$, and
any distinct points $y_1,y_2 \in \{y \in X :~
d(T^kx,T^ky)<\frac{\epsilon}{2} \forall |k|>n\}$ there exists $k \le
n$  for which $d(T^kx,T^ky)\ge \epsilon$.
Thus $\T(x)$ is a countable union of %$T^k$-translates of
sets which are $\epsilon$-separated according to the metric
$$ d_n(x,y):= \max_{|k| <n} d(T^kx,T^ky)$$
 By compactness of $(X,d)$, $(X,d_n)$ is also compact, and so  any
 $\epsilon$-separated set is finite. It follows that $\T(x)$ is at most countable.
\end{proof}

In a similar manner, the Gibbs relation of a $\ZZ^d$ action is
defined as those pairs of points with orbits whose orbits are
asymptotically converging.
% (one can further generalize to a
%topological action of a locally compact groups, but we will not
%require such generality)

A $\T_X$-\emph{holonomy} is a Borel isomorphism $\varphi:A \to B$
with $A,B \subset X$ Borel sets, such that $(x,\varphi(x)) \in \T_X$
for every $x \in A$.

%A collection $\mathfrak{C}$ of $\T_X$-holonomies generates $\T_X$ if
%for any $(x,y) \in \T(X)$ there exist $f \in \mathfrak{C}$ with
%$y=f(x)$.

In the rest of this paper $X$ will be a either a one-dimensional or
$d$-dimensional \emph{subshift} and $T:X \to X$ will denote the
shift map or shift action of $\ZD$ respectively.
% A subshift is a compact subset of the cantor set $S^\ZZ$, invariant under the shift map $T(x)_n=x_{n+1}$,

%In view of the above lemma,
$\T_X$ is a countable, standard Borel equivalence relation in the
sense of Feldman and Moore \cite{feldman_moore_77_i}. It follows
that there exists a countable group $\Gamma$ of Borel automorphisms
of $X$ which generate $\T$, meaning  $\Gamma x = \T(x)$ for all $x
\in X$. In general, such $\Gamma$ can not be chosen as a group of
homeomorphisms of $X$.

%For a Borel set $A \subset X$, denote the \emph{$\T(X)$-saturation}
%of $A$ by:
%$$\T(A) =\{ y \in X: ~ \exists \, x \in A ~ (x,y) \in \T_X\}$$
%Evidently, $\T(A)$ is a Borel set whenever $A$ is.

Let
$$\mathcal{F}(X)= \{ \varphi \in \Homeo(X):~ \exists N >0\mbox{ s.t. } \forall |k|> N\, \forall x \in X~\varphi(x)_k = x_k\}$$
The set $\mathcal{F}(X)$ is a countable group of Homeomorphisms. In
Krieger's terminology from \cite{kreiger_dimension_1980},
$\mathcal{F}(X)$ is the group of ``uniformly finite-dimensional
bijections''.

Here is a convenient countable set of generators for this group:
\begin{lem}\label{lem:swap_inv_generate}
$\mathcal{F}(X)$ is generated by involutions of the form
$$\xi_{a,b}(x)_k=\left\{ \begin{array}{cc}
a_k & \mbox{if }k\in B_n \mbox{ and } x\in [b_k]_{B_n}\\
b_k & \mbox{if }k\in B_n \mbox{ and } x\in [a_k]_{B_n}\\
x_k & \mbox{otherwise}    \end{array} \right. ,$$ where $a,b \in
\Sigma^{B_n}$ for some $n \in \NN$ are such that $\xi_{a,b}(x) \in
X$ for every $x \in X$.
\end{lem}
\begin{proof}
Choose $\varphi \in \mathcal{F}(X)$. By continuity of $\varphi$ and
compactness of $X$, it follows that there exists $N \in \NN$ such
that $\varphi(x)_{[-N,N]}$ depends only on $x_{[-N,N]}$.
  Fix $N \in \NN$. Denote by
  $\mathcal{F}_N(X)$ the subgroup of $\mathcal{F}(X)$ for which the
  above holds with this given $N$. We have $\mathcal{F}(X) = \bigcup_{N \in \NN}\mathcal{F}_N(X)$.

We can define an injective group homomorphism $\alpha$ from
$\mathcal{F}_N(X)$ into the group of permutations of the finite set
$\Sigma^{B_N}$ by considering the action of $\varphi$  on the
coordinates inside $B_N$.  Any group of permutations on a finite set
is generated by ``swap'' involutions. Applying $\alpha^{-1}$ on such
a ``swap'' involution, we obtain an  involution of the form
$\xi_{a,b}$.
\end{proof}

 For every $a,b,c \in \Sigma^{B_n}$,
$\xi_{a,b}\xi_{b,c}=\xi_{a,c}$  and the right-hand-side is defined
whenever the left-hand-side is defined.

The group $\mathcal{F}(X)$ generates a sub-relation of $\T_X$. We
denote by $\T^0_X$ the orbit relation of $\mathcal{F}(X)$, and refer
to it as the \emph{topological Gibbs-relation}.

Call a point $x \in X$ $\T$-\emph{regular} if for any $x' \in X$,
$(x,x') \in \T_X$ implies $(x,x') \in \T^0_X$. A $x \in X$ point
which is not $\T$-regular, is called a \emph{$\T$-singularity}.

It is a simple observation that in case $X$ is a subshift of finite
type, $\T^0_X=\T_X$. In sections \ref{sec:kalikow}-\ref{sec:dyck} we
will see some examples where this is not the case.
 % $\mathcal{F}$ is an amenable group, since it consists only of elements of finite order.

\subsection{Functions with $d$-summable variation}
For $n=(n_1,\ldots,n_d) \in \ZD$ we denote% the supremum norm of the vector $n$ by
 $\|n\|=\max_{1\le i\le d}n_i$. For $f:X \to
\mathbb{R}$, let
$$v_k(f) := \sup\{|f(x)-f(y)|:~x,y \in X,\, x_{n}=y_{n}, ~\forall\, \|n\|\le k\},$$
and $v_0(f)=\|f\|_{\infty}$. The $d$-sum-of-variations norm of $f:X
\to \mathbb{R}$ is defined by:
$$\|f\|_{\SVD}:= \sum_{k=1}^{\infty}k^{d-1} v_{k-1}(f)$$
If $\|f\|_{\SVD} < \infty$ we say that $f$ has \emph{$d$-summable
variation}, as in \cite{schmidt_super_k}. Denote by $\SVD(X)$ the
collection of real-valued function on $X$ with $d$-summable
variation. $\SVD(X)$ is a separable Banach-space with respect to the
norm $\|\cdot \|_{\SVD}$.

A function $f \in \SVD(X)$ defines a $\T_X$-cocycle $\phi_f:\T_X \to
\mathbb{R}_+$  by:
$$\phi_f(x,y)=\exp\left(\sum_{n \in \ZD}f(T^nx)-f(T^ny)\right)$$

For $(x,y) \in \T_X$ there exist $k>0$ such that $x_n=b_n$ for $n
\in \ZD$ with $\|n\|>k$. It follows that $|f(T^nx)-f(T^ny)| \le
v_{\|n\|-k}$. Since the number of there are order of $j^{d-1}$
points in $\ZZ^d$ with norm $j$, it follows that $\phi_f(x,y)$ is
well-defined whenever $f \in \SVD(X)$ and $(x,y) \in \T_X$.

\subsection{Conformal and Gibbs measures}
In the following we recall some terminology on conformal measures.
For further details and references see
\cite{aaro-nakada-2007,feldman_moore_77_i,petersen_schmidt97}.

Let $\mathcal{R} \subset X \times X$ be a Borel-equivalence relation
on $X$. A measure $\mu \in \mathcal{P}(X)$ is
$\mathcal{R}$-nonsingular if $\mu(A)=0$ implies
$\mu(\mathcal{R}(A))=0$ for any Borel set $A \subset X$.

 If $\mu$ is $\T_X$-nonsingular the \emph{Radon-Nikodym cocycle} of $\mu$ with
respect to $\T$ is a measurable map $D_{\mu,\T}:\T \to \mathbb{R}_+$
satisfying $\frac{d\mu \circ f}{d\mu}(x)=D(x,f(x))$ for any
$\T_X$-holonomy $f$. $D_{\mu,\T}$ is uniquely defined up to a
$\mu$-null set.
%$D_{\mu,\T}$ is
%a $\T$ cocycle, as it satisfies
%$$D(x,z)=D(x,y)D(y,z),$$
%whenever $(x,y),(y,z) \in \T$.

For a measurable cocycle $\phi:\mathcal{R} \to \mathbb{R}_+$, a
measure $\mu\in \mathcal{P}(X)$ is
\emph{$(\phi,\mathcal{R})$-conformal} if it is
$\mathcal{R}$-nonsingular with $\log D_{\mu,\mathcal{R}} = \phi$.
Observe that for any countable group $\Gamma \subset
\mathit{Aut}(X)$ which generates $\mathcal{R}$, a measure $\mu$ is
$(\phi,\mathcal{R})$-conformal iff $\frac{d\mu \circ
\gamma}{d\mu}(x)=\exp\left(\phi(x,\gamma x)\right)$ on a set of full
$\mu$-measure, for any $g \in \Gamma$.

Call a measure $\mu\in \mathcal{P}(X)$ a \emph{Gibbs measure of $f
\in \SVD(X)$} if it is $(\phi_f,\T_X)$-conformal.

\begin{lem}\label{lem:svd_cocycle_converge}
  If $f,g \in \SVD$ and $x,y \in X$ and $x_{B_r^c}=y_{B_r^c}$, then $$|\phi_f(x,y)- \phi_g(x,y)| \le
  C_r\|f-g\|_{\SVD},$$
  with $C_r = 100\cdot(2r)^{d}$.
\end{lem}
\begin{proof}
  Observe that $\phi_f-\phi_g = \phi_{f-g}$, so assume without loss of generality that $g=0$.
  $$\phi_f(x,y) = \sum_{k} (f(T^kx)-f(T^ky))$$
  For $|k| \le r$, we have $|f(T^kx)-f(T^ky)| \le 2\|f\|_{\infty}$.
  Now if $x_{B_r^c}=y_{B_r^c}$, it follows that for $k \in \ZD$ with $|k| > r$,
  $$|f(T^kx)-f(T^ky)|< v_{|k|-r}(f)$$
  Summing over all $k \in\ZD$, we get:
  $$\sum_{k \in Z}|f(T^kx)-f(T^ky)| =  \sum_{j=0}^{\infty}\sum_{|k| = j}  |f(T^kx)-f(T^ky)| \le$$
  $$ \le 2|B_r|\cdot \|f\| \sum_{j=r+1}^\infty N_{j} v_{j-r}(f)$$
  where $N_j$ is the number of $k$'s in $\ZD$ with $|k|=j$, and $B_r = \{ k \in \ZD : |k| \le r\}$. As $N_j\le \frac{1}{2}C_r j^{d-1}$,
  and $|B_r| \le \frac{1}{2}C_r$, this completes the proof.
\end{proof}

The following property makes $\SVD(X)$  a suitable Banach-space to
study $\T^0$-conformal measures:

\begin{prop}\label{prop:limit_of_gibbs}
  Suppose $\{f\}_n$ are sequence of functions in $\SVD(X)$ which converge to $f$ in norm,
  denote $\phi_n:=\phi_{f_n}$ and $\phi:=\phi_f$.
  If $\mu_n$ is $(\phi_{n},\T^0)$-conformal, and $\mu_n$ tends weakly to $\mu$, then $\mu$ is $(\phi,\T^0)$-conformal.
\end{prop}
\begin{proof}
The statement will follow once we show that
$$\int_{[a]}f(x) d\mu(x) = \int_{[b]}f(x)\phi(x,\xi_{a,b}(x))d\mu(x)$$
whenever $\xi_{a,b}$ and $\xi_{b,a}$ are defined and every
continuous $f:X \to \mathbb{R}$.

The assumption that $\mu_n$ is $(\mathcal{F},\phi_{n})$-conformal
implies that
$$\int_{[a]}f(x) d\mu_n(x) = \int_{[b]}f(x)\phi_n(x,\xi_{a,b}(x))d\mu_n(x)$$
whenever $a,b$ and $f$ are as above. Since $\mu_n$ tend weakly to
$\mu$,
$$\lim_{n \to \infty}\int_{[a]}f(x) d\mu_n(x) =\int_{[a]}f(x) d\mu(x).$$
By lemma \ref{lem:svd_cocycle_converge}, since $\lim_{n \to
\infty}\|f_n-f\|_{\SVD}=0$, the functions
$g_n(x):=f(x)\phi_n(x,\xi_{a,b}(x))$ converge uniformly on $[b]$ to
$g(x):=f(x)\phi(x,\xi_{a,b}(x))$. Thus,
$$\lim_{n \to \infty}f(x)\int_{[b]}\phi_n(x,\xi_{a,b}(x))d\mu_n(x) =\int_{[b]}f(x)\phi(x,\xi_{a,b}(x))d\mu(x)$$
\end{proof}

\subsection{Pressure and Equilibrium}
% refs: ruelle, walters ...
%Suppose $(X,T)$ is a $\ZD$ topological dynamical system ($X$ is a
%compact space and $T$ is a $\ZD$ action on $X$ by homeomorphisms).
Suppose $\phi:X \to \RR$ is a continuous function (regraded as a
``potential'' on $X$), $\mathcal{U}$ a finite open cover of $X$, and
$F \subset \ZD$ is a finite set. Define a \emph{partition function}:
\begin{equation}
Z_F(\phi,\mathcal{U})= \min\{ \sum_{u \in
\mathcal{U}'}\exp\left[\sup_{x \in u} \sum_{n \in
F}\phi(T_nx)\right] :~ \mathcal{U}' ~\mbox{is a subcover of}~
\mathcal{U}^F\}
\end{equation}

The topological pressure of $\phi$ with respect to $\mathcal{U}$ is
defined as:
$$ P(\phi,\mathcal{U})= \lim_{n \to \infty}|F_n|^{-1}\log
Z_{F_n}(\phi,\mathcal{U}),$$ where $F_n=[1,n]^d$.

 The topological pressure of
$\phi$ is obtained by taking supremum over all finite open covers:
$$P(\phi)= \sup_{\mathcal{U}}P(\phi,\mathcal{U})$$

%If $(X,T)$ is expansive and $\diam\mathcal{U}$ is an expansive
%constant, then $P(\phi,\mathcal{U})=P(\phi)$. In particular,  when
%$X \subset \Sigma^{\ZD}$ is a subshift and $T$ is the shift action,
%the cover $\mathcal{U}_0=\{[a]_0 : ~ a \in \Sigma\} $ satisfies this
%condition, and so:
Concretely,

$$P(\phi)= \lim_{n \to \infty}|F_n|^{-1}\log \sum_{a \in
L_n(X)}\exp\left[\sup_{x \in [a]} \sum_{n \in F_n}\phi(T_n x)
\right]$$

For an invariant measure $\mu \in \mathcal{P}(X,T)$, the
\emph{measure theoretic pressure} is defined by: $$P_\mu(\phi)=
h_\mu(X,T)- \int \phi(x) d\mu(x)$$ where $h_{\mu}(X,T)$ is the
Kolmogorov (measure-theoretic) entropy of  $(X,\mu,T)$.

The \emph{variational principal} is a theorem which relates
measure-theoretic pressure with the topological one:

\begin{thm*}\textbf{(The variational principal for pressure)}
For any continuous function $f:X \to \mathbb{R}$,
  $$\sup_{\mu \in \mathcal{P}(X,T)}P_{\mu}(f) = P(f)$$
\end{thm*}

See \cite{walters_var} for a proof.

A measure $\mu \in \mathcal{P}(X,T)$ is called an \emph{equilibrium
state} for $f$ if $P_\mu(f)=P(f)$. Whenever $T$ is expansive and $f$
continuous, the existence of an equilibrium state is assured.

If $f,g: X \to \mathbb{R}$ are continuous functions such that $h:=
f- g$ has integral zero with respect to any measure in
$\mathcal{P}(X,T)$, then the sets of equilibrium measures of $f$ and
$g$ coincide. In particular this is the case if $f$ and $g$ are
cohomologous.

\section{\label{sec:dobru} Dobru{\v{s}}in's theorem and the Lanford-Ruelle theorem}

For a measure $\mu$ to be a Gibbs measure is a ``local property'':
it imposes a condition on the Radon-Nikodym cocyle for pairs of
points $(x,y) \in \T$. On the other hand, for $\mu$ to be
equilibrium measure is a ``global property'': An equilibrium $\mu$
must maximize the pressure, which is a global quantity.

The following we state theorems by Lanford and Ruelle and
Dobru{\v{s}}in specify a framework within which these global and
local notions coincide:

In order to state  Dobru{\v{s}}in's theorem we introduce the
following definition:  A subshift $X \subset \ZD$ of  satisfies
\emph{condition $(D)$} if there exist increasing sequences of
integers $\{n_k\}_{k=1}^{\infty}$ and $\{m_k\}_{k=1}^{\infty}$ with
$n_k < m_k$ and $\lim_{k \to \infty}\frac{m_k}{n_k}=1$ such that for
any $x,y \in X$ and $k \in \NN$ there exist $z \in X$ with
$z\mid_{F_{n_k}}=x\mid_{F_{n_k}}$ and $z \mid_{\ZD \setminus
F_{m_k}}=y\mid_{\ZD \setminus F_{m_k}}$.

An SFT $X$ is \emph{strongly irreducible} if it satisfies condition
$(D)$ with $m_k=n_k+L$ for some integer $L$. In the case $d=1$, an
SFT satisfies  condition $(D)$ iff it is strongly irreducible, iff
it is mixing.

Dobru{\v{s}}in's theorem states the following:
\begin{thm*}|\label{thm:dobrusin}(\textbf{Dobru{\v{s}}in \cite{dobrusin_gibbs}})
Let $X \subset S^\ZD$ be a subshift which satisfies condition $(D)$,
and $f:X \to \mathbb{R}$ a function with summable variation. Then
any translation invariant Gibbs state is an equilibrium for $f$.
\end{thm*}

On the other direction, there is the following theorem of Lanford
and Ruelle:
\begin{thm*}\label{thm:lanford_ruelle}(\textbf{Lanford-Ruelle \cite{lanford_ruelle}})
Let $X \subset \Sigma^{\ZD}$ be a subshift of finite type, and $f
\in \SVD(X)$. Then any equilibrium measure for $f$ is a  Gibbs
measure for $f$.
\end{thm*}

As explained in the introduction, we wish to check the validity of
these theorems for subshifts which are not of finite type. Without
any restrictions on the subshifts, both the above theorems fail to
generalize for various reasons. Here is a simple example of a
subshift which admits an equilibrium which is not Gibbs:

Let $\Sigma=\{0,1,2\}$. We define the subshift $X \subset
\Sigma^{\ZZ}$ by the condition that for any $x \in X$, $n \ge 2$ and
$k \in \ZZ$,
$$\#\{j \in [k,k+2^n]:~ x_j=0 \} \le n.$$

It follows that any for any translation invariant probability
measure $\mu \in \mathcal{P}(X)$, $\mu{[0]}=0$. Thus,
$\htop(X)=\log(2)$ and the unique measure of maximal entropy is the
symmetric Bernoulli measure on $\{1,2\}^{\ZZ}$, which we denote by
$\mu_0$. To see that $\mu_0$ is not $\T$-invariant, note that if $x
\in \{1,2\}^\ZZ \subset X$ ,then replacing a single coordinate with
$0$ leaves us with an admissible point in $X$. This defines a
$\T$-holonomy $g:\{1,2\}^\ZZ \to A$ where $A=\{x \in X:~ x_0=0,~
x_i\ne 0 \forall i\ne 0\}$. Since $\mu_0(A)=0$ and
$\mu_0(\{1,2\}^\ZZ)=1$, we see that $\mu_0$ is singular with respect
to $\T_X$. Thus, $\mu_0$ is an equilibrium for a constant function,
yet is not a Gibbs measure.

%
%This subshift %does not satisfy Dobru{\v{s}}in's condition. However,
%is topologically mixing and the relation $\T_X$ is minimal, yet it
%is pathological in a way.

An attempt to find a statement which generalizes Lanford-Ruelle and
holds for an arbitrary subshift leads us to the following theorem:
\begin{thm}\label{thm:local_lanford_ruelle}
Let $X \subset \Sigma^{\ZD}$ be a subshift, and $f \in \SVD(X)$.
Then any equilibrium measure $\mu$ for $f$ is
$(\phi_f,\T^0)$-conformal.
\end{thm}

\begin{proof}
The proof we bring here combines some elements form Burton and
Steif's proof of the corresponding theorem on measures of maximal
entropy for SFTs in \cite{burton_steif94}, and other ingredients
from the proof of the Lanford-Ruelle appearing in
\cite{ruelle_themo}.

The first step is a reduction of the theorem to ``local functions''
$f:X \to \mathbb{R}$ - this means that $f(x)$ depends only on $x
\mid_F$ for some set finite $F \subset \ZD$:

 The collection of local
functions $\mathit{Loc}(X)$, is dense in $\SVD(X)$. $\SVD(X)$ is a
separable Banach space, and the pressure function $P:\SVD(X) \to
\mathbb{R}$ is convex and continuous.

% it follows from a theorem of
%Mazor that there is a residual set $\mathcal{U} \subset \SVD(X)$
%where $P$ has a unique tangent functional- which is the unique
%equilibrium for $f$.
%
Assume we know that for any local function equilibrium measures are
$\T^0$-conformal.
%, for any $f \in \mathcal{U}$ we find a sequence
%$f_n \in \mathit{Loc}(X)$, such that $f_n \to f$ in $\SVD(X)$ (thus
%in $C(X)$), and a sequence of measures $\mu_n$ s.t. $\mu_n$ is an
%$f_n$-equilibrium (thus $f_n$-conformal). It follows that $\mu_n\to
%\mu$ weakly, where $\mu$ is the unique equilibrium for $f$. It
%follows that $\mu$ is $(f,\T^0)$-conformal. To extend the proof for
%an arbitrary $f \in \SVD(X)$, apply
A theorem of Lanford and Robinson from \cite{lanford_robinson_68},
states that for a continuous convex function on a separable Banach
space $X$ , and a dense set $X_0 \subset X$,  any tangent functional
at $x \in X$ is in the weak-closure of the convex hull of the set
$$\{\lim_{n\to \infty}y_n:~ y_n \mbox{ is a tangent at } x_n \in X_0,~ x_n \to x
\}$$
% of limits $y=\lim y_n$ whenever $y_n$ is the unique tangent at
%$x_n$ and $x_n \to x$.
Thus, any equilibrium $\mu$ for $f \in \SVD(X)$ is a limit of
$\mu_n$ which are equilibrium for local $f_n$'s, such that
$\|f_n-f\|_{\SVD}\to 0$. Assuming the proposition holds for local
functions, each $\mu_n$ is $f_n$-conformal, and so by proposition
\ref{prop:limit_of_gibbs} $\mu$ is $(f,T^0)$-conformal.

The rest of the proof is establishes the result for local functions
$f:X \to \mathbb{R}$. A further reduction is to assume that $f$ is a
\emph{site potential}, meaning $f(x)$ depends only on $x_0$. This is
no loss of generality, since for any local function $f$ there is an
isomorphism of $X$ which recodes $\Pi(x)_0=[x]_{F_n}$ for
sufficiently large $n$, which maps $f$ onto a site potential
$\overline{f}$, maps any conformal measure for $f$ onto a conformal
measure for $\overline{f}$, and an $f$-equilibrium onto an
equilibrium for $\overline{f}$.

%Suppose $\mu$ is an equilibrium measure for a local function $f$.
Introduce an increasing sequence of sub-relations of $\T^0_n \subset
\T^0$, such that $\T^0= \bigcup_{n \ge 1} T^0_n$: $\T^0_n$ is the
orbit relation of
$$\mathcal{F}_n(X)= \{ g \in \Homeo(X):~  \forall |k|> n \forall x \in X~g(x)_k = x_k\}.$$
We will prove that $\mu$ is $(\phi_f,\T^0_n)$-conformal, for each
$n$, thereby show  $\mu$ is $(\phi_f,\T^0)$-conformal:

We begin by proving that  $\mu$ is $(\phi_f,\T^0_0)$-conformal. For
$a \in \Sigma$, let $\overline{a} \subset \Sigma$ denote the
equivalence class of $a$ under the relation $\sim_X$ spanned by
$\mathcal{F}_X$. Let $\overline{\Sigma}=\{ \overline{a} : a \in
\Sigma\}$. The map $\pi:X \to \overline{\Sigma}^\ZD$ defined by
$\pi(x)_n = \overline{x_n}$ is a factor map onto
$\overline{X}=\pi(X)$. Observe that for any $y \in \overline{X}$,
$$\pi^{-1}(y)=\{x \in \Sigma^{\ZD}: ~ x_n \in y_n\,  \forall n \in \ZD \}$$

Define $\mu_0 \in \mathcal{P}(X)$ by setting
$$\mu_0 \pi^{-1}= \mu \circ \pi^{-1}$$ and

$$\mu_0([x]_F \mid \pi(x)) = \prod_{n \in F}p(x_n),$$
where
\begin{equation}\label{eq:gibbs_prob}
p(a)=\left(\sum_{b \sim a}\exp f(b)\right)^{-1}\exp f(a),~ a\in
\Sigma
\end{equation}
By definition, the probability $\mu_0$ is defined so that  the
coordinates of a point $x$ are relatively independent over $\pi(x)$,
with probabilities proportional to $\exp(f(x_n))$.

Let us compare $P_{\mu}(f)$ and $P_{\mu_0}(f)$:

$$P_{\mu}(f)=h_{\mu}(X)+\int f(x)d\mu(x)=$$
$$=h_{\mu\circ \pi^{-1}}(\overline{X})+ h_{\mu}(X \mid \pi) +\int f(x)d\mu(x)=$$
$$=h_{\mu_0\circ \pi^{-1}}(\overline{X}) + h_{\mu}(X \mid \pi) +\int f(x)d\mu(x)$$

Now $h_{\mu}(X \mid \pi) \le \int H_{\mu(\cdot \mid \pi(x)_0)}(x_0)
d\mu(x)$, with equality holding iff there is relative independence
of the coordinates for $\mu$ given the projection $\pi$, and $x_n$
in independent of $\pi(x)$ given $\pi(x)_n$. Also, for every
$\overline{a} \in \overline{\Sigma}$,
$$H_{\mu(\cdot \mid \pi(x)_0=\overline{a})}(x_0) + \int_{\pi^{1}[\overline{a}]_0\}}\phi(x)d\mu(x) \le H_{\mu_0(\cdot \mid \pi(x)_0
=\overline{a})}(x_0) +
\int_{\pi^{-1}[\overline{a}]_0}\phi(x)d\mu_0(x)$$

We conclude that $P_{\mu_0}(f) \ge P_{\mu}(f)$ with equality iff
$\mu=\mu_1$. The measure $\mu_1$ was defined so that $\frac{d\mu
\circ f_{a,b}}{d\mu}(x)=\log f(x,f_{a,b}(x))$ for $a \sim b$ $a,b\in
\Sigma$. Thus, $\mu$ is $\T^0_1$-conformal.

To prove that $\mu$ is $\T^0_n$ conformal for $n >1$, repeat the
previous argument, combined with the following property of
equilibrium measures for actions of sub-lattices of $\ZD$: let
$$A_n(f):=\frac{1}{|F_n|}\sum_{k \in F_n}f\circ T_k,$$ then $\nu$ is
an equilibrium for the function $A_n(f)$, with respect to
translations by the sublattice  $n\ZD$ iff $A_n(\nu)$ is an
equilibrium for $f$ with respect to translations in  $\ZD$.
Furthermore, if $\mu$ is invariant with respect to translations in
$\ZD$, $A_n(\mu)=\mu$.

Now define $\mu_n$ to be a $\T^0_n$-conformal measure, obtained from
$\mu$ similarly to $\mu_0$, except that  every $n$-cube
configuration is recoded into one symbol, and consider translations
by $n\ZD$. It follows that the pressure of $\mu_n$ with respect to
$A_n(f)$ is greater or equal to that for $\mu$, with equality iff
$\mu_n=\mu$. This proves that $\mu$ is $(\phi_f,\T^0_n)$-conformal.
\end{proof}

Here is a direct corollary of this result:
\begin{cor}\label{cor:reg_equilibrium_Gibbs}
  Let $X \subset S^\ZD$ be a subshift, $f \in \SVD(X)$, and $\mu$ an
  $f$-equilibrium. If the support of $\mu$ contains no
  $\T$-singularities, then $\mu$ is a Gibbs measure for $f$.
\end{cor}
\begin{proof}
Let $X$, $f$ and $\mu$ satisfy the conditions above. Let $\varphi:X
\to X$ be a $\T$-holonomy. The assumption that $\mu$'s support
contains no $\T$-singularities is equivalent to the existence of a
$\T$-saturated Borel set $X_0 \subset X$ with $\mu(X_0)=1$ and
$(\T_X \cap X_0)= (\T^0_X \cap X_0)$. Let $\Gamma_0$ be a countable
group of $\T$-holonomies, which fix all points in $X_0$. Thus,
$\Gamma_0$ together with $\mathcal{F}(X)$ generate $\T(X)$. The
elements of $\Gamma_0$ are equal to the identity modulo $\mu$. By
theorem \ref{thm:local_lanford_ruelle} above, $\mu$ is
$(\T^0,f)$-conformal. It follows that $\frac{d\mu \circ
\gamma}{d\mu}(x)=\exp\left(\phi_f(x,\gamma x)\right)$ for all
$\gamma \in \langle \Gamma_0, \mathcal{F}(X) \rangle$. Thus, $\mu$
is $(\phi_f,\T)$-conformal.
\end{proof}

%\section{\label{sec:one-sided}One-sided subshifts, and the tail-relation}

\textbf{Remark}: Observe that Dobru{\v{s}}in's theorem and
Lanford-Ruelle theorem are valid for SFTs in any dimension $d \ge
1$. When $d=1$, there is in fact a unique equilibrium for any $f \in
\SVD(X)$, when $X$ is an SFT. In this case, the unique equilibrium
is also the unique translation invariant Gibbs measure. One approach
for proving this is via one-sided subshifts and the Ruelle operator,
as in Bowen \cite{bowen_equlibrium} and Walters
\cite{walters_g_func}. If the $X$ is an irreducible $1$-dimensional
SFT, there is also a unique Gibbs measure.

The rest of this paper is dedicated to some examples of equilibrium
and Gibbs measures for some subshifts which are not of finite type.

\section{\label{sec:kalikow} Kalikow-type subshifts}
In this section we study $\T$-invariant measures and measures of
maximal entropy for a family of subshifts obtained by ``a random
walk with random scenery''. We call these Kalkow-type subshifts, in
homage to Kalikow's  paper about the $T$-$T^{-1}$ transformation
\cite{kalikow82}.

Let $T:X \to X$ be a homeomorphism of a compact space $X$. The $T-
T^{-1}$ transformation is the skew-product $\widehat{T}:
\{-1,1\}^\ZZ \times X \to \{-1,1\}^\ZZ \times X$ given by
$\widehat{T}(x,y)=(\sigma x, T^{x_0}y)$, where $\sigma$ is the shift
map on $\{-1,1\}^\ZZ$. The map $\widehat{T}$ is a homeomorphism of
$\{-1,1\}^\ZZ \times X$. We will make use of results of Marcus and
Newhouse \cite{marcus_newhouse_78}, which  describe measures of
maximal entropy for skew products of this form.

Let us restrict to the case where  $X \subset \Sigma^\ZZ$ is a
subshift, and $T:X \to X$ is the shift. In this case, a subshift
$\hX \subset (\{-1,1\}\times\Sigma)^\ZZ$ appears naturally as a
factor of $(\{+1,-1\}^\ZZ\times X,\widehat{T})$, as the image of the
map $\pi: \{-1,1\}^\ZZ \times X \to \{-1,1\}^\ZZ \times \Sigma^\ZZ$
given by: $\pi(x,y)_n=(\widehat{T}^n(x,y))_0$. Call $\hX$ the
\emph{Kalikow-type subshift} associated with $X$. Observe that the
map $\pi$ is injective on a dense orbit: Namely, it is injective
when restricted to the following dense subset:
$$\{(w,n) \in \{+1,-1\}^\ZZ\times X:~ \inf_{m<M}\sum_{k=m}^Mw_k=-\infty ,~
\sup_{m<M}\sum_{k=m}^Mw_k=+\infty\}.$$

Let us describe the admissible words for $\hX$. First, we define
$\Phi:\{+1,-1\}^* \to \ZZ$ by
$$\Phi(x_1,\ldots,x_n)=\sum_{i=1}^{n}x_i$$ The subshift $\widehat{X}
\subset (\{-1,1\}\times\Sigma)^\ZZ$ is the set of all sequences
$\left(\ldots,(x_{-1},y_{-1}),(x_0,y_0),(x_1,y_1),\ldots\right)$
with the following restrictions:
\begin{enumerate}
\item{For all $i,j \in \ZZ$, with $i \le j$ if $\Phi(x_i,\ldots,x_j)=0$ then $y_i=y_j$.}
\item{For any finite subset $I \subset \ZZ$ with $\min I = m$, there exist $z \in X$
such that $z_{s(i)}=y_{i}$ for all $i \in I$, where
$s(i)=\Phi(x_m,\ldots,x_i)$ .}
\end{enumerate}

The simplest case of the above construction is when $X=\{0,1\}^\ZZ$
is the $2$-shift. For $p \in (0,1)$ let $\mu_p$ be the measure on
$\hX$ defined by projecting via $\pi$ the product measure of the
$(p,1-p)$ i.i.d product on $\{0,1\}^\ZZ$ with the symmetric product
measure on the $2$-shift $X$. The measure theoretic entropy of the
shift on $\hX$ with respect to the measure $\mu_p$ is $|2p-1|\log 2
-p\log p - (1-p)\log(1-p)$. A simple calculations shows that this
expression is maximized when $p=0.2$ or $p=0.8$. These are the only
ergodic measures of maximal entropy for this subshift.

The entropy of $(\hX,\mu_{\frac{1}{2}},\sigma)$ is easily shown to
be $\log(2)$. As we will see, the measure $\mu_{\frac{1}{2}}$ is
$\T_2(\hX)$-invariant, but is not a measure of maximal entropy for
the subshift $\hX$.

By using the Rokhlin-Abramov formula for the entropy of a
skew-product, Marcus and Newhouse obtained the following entropy
calculation:

\begin{prop}\label{prop:mu_p_entropy}
Let $X=\{0,1\}^\ZZ$, $p \in (0,1)$ and $\mu_p$ be the measure on the
Kalikow-type subshift $\hX$ defined above. Then:
$$h_{\mu_p}(\hX)=|2p-1|\log2+H(p)$$
where $H(p)=p\log(p)+(1-p)\log(1-p)$.
\end{prop}

Furthermore, by solving a variational problem using and the
Rokhlin-Abramov formula, Marcus and Newhouse obtain the following
proposition:
%\begin{prop}\label{prop:maximal_measures_hX}
%Let $X=\{0,1\}^\ZZ$ and $\hX$ as above.
%  $\htop(\hX)=\log(\frac{5}{2})$ and the egodic measures of
%  maximal entropy are $\mu_{0.8}$ and $\mu_{0.2}$.
%\end{prop}

\begin{prop} \label{prop:t_t_inverse_entropy_formula} Let $X$ be any subshift with
topological entropy $\log t$. Let $\hX$ be the Kalikow-type subshift
associated with $X$. Then the topological entropy of the subshift
$\hX$ is
$$\htop(\hX)=\log(\frac{t^2+1}{t})$$
%Any measure of maximal entropy is obtained by projecting  *********
\end{prop}

Starting with any subshift $X$ and using arguments from
\cite{marcus_newhouse_78}, it is relatively simple describe all
measures of maximal entropy for $\hX$ in terms of the measures of
maximal entropy of $X$. For example, in case
$X=\{0,1,\ldots,N\}^\ZZ$, by maximizing the expression in
proposition \ref{prop:mu_p_entropy}, we see that
$\htop(\hX)=\log(\frac{N^2+1}{N})$ and the egodic measures of
maximal entropy are $\mu_{p}$ with $p=\frac{N^2}{1+N^2}$ or
$p=\frac{1}{1+N^2}$.

The following proposition is a partial analog of the Lanford-Ruelle
theorem for Kalikow-type subshifts, with respect to
$f=\mathit{const}$.
\begin{prop}\label{prop:kalikow_lanford_ruelle}
Let $X=\{1,\ldots,N\}^\ZZ$ be a full-shift,  and let $\hX$ be the
associated Kalikow-type subshift. All measures of maximal entropy
for $\hX$ are $\T$-invariant.
\end{prop}
\begin{proof}
By Marcus and Newhouse's result above, the ergodic measures of
maximal entropy $\mu_+$ and $\mu_-$ are obtained by taking
  the product of $(p,q)$ -bernoulli measure on the base and Haar measure on the ``scenery'',
where $p=\frac{N^2}{1+N^2}$ or $p=\frac{1}{1+N^2}$ respectively. We
will prove the proposition for $\mu_{+}$. The proof for $\mu_{-}$ is
symmetric.

Let
$$\hX_{+}=\{z \in \hX:~ \lim_{n \to +\infty}\Phi_n(z)=+\infty ~, \lim_{n \to
-\infty}\Phi(z)=-\infty\},$$ where:
$$\Phi_k((x_n,y_n)_{n \in \ZZ}) = \left\{\begin{array}{cc} \sum_{j=0}^{k-1}x_j & k>0\\
0 & k=0 \\
-\sum_{j=k}^{-1}x_k & k<0 \end{array} \right.$$

$\hX_{+}$ is a Borel subset which is saturated with respect to
$\T_{\hX}$. Also, for any $p>\frac{1}{2}$ $\mu_{p}(\hX_{+})=1$. In
particular, $\mu_+(\hX_+)=1$.

Next, we describe a collection of $\T_{\hX_{+}}$-holonomies
$\mathfrak{C}_+$ which generates the Gibbs relation:  For $k,n \in
\ZZ$ define the following Borel set:
$$A_{k,n}=\{ z \in \hX:~ \forall j<k \,\Phi_{j}(z)<\Phi_k(z) ~ ,
\, \forall j > k+n \, \Phi_{j}(z)>\Phi_{k+m}(z)\}$$

Let $w_1,w_2$ be an admissible words for $\hX$, with $|w_i|=n$,
$\Phi(w_i)=m_i>0$ and $0 \le \Phi((w_i)_0^j) \le \Phi(w)$ for all $0
\le j \le n$ and $i=1,2$. The words $w_1,w_2$ represent
``excursions'' of length $n$ which terminate in the rightmost
coordinate in $m$. For any $x \in A_{k,n}$, changing the coordinates
from $k$ to $k+n$ for $w_1$ to $w_2$ yields an admissible sequence
in $A_{k,n}$, since by definition of $A_{k,n}$ the ``scenery''
visited in the time-interval $[k,k+n]$ is never visited outside this
time-interval. This defines a $\T_{\hX}$-holonomy
$g_{w_1,w_2;k}:([w_1]_k \cap A_{k,n}) \to ([w_2]_k \cap A_{k,n})$.
The collection $\mathfrak{C}_{+}$ consists of all holonomies of this
form.

We now show that $\mathfrak{C}_+$ generates $\T_{\hX_{+}}$. Let
$(x,y) \in \T_{\hX_{+}}$. Since $\lim_{n \to + \infty}\Phi(x)=
+\infty$ and $\lim_{n \to + \infty}\Phi(x)= +\infty$, there exist
infinitely many $M,N >0$ such that $\Phi_j(x)>\Phi_{M}(x)$ for all
$j
> M$, $\Phi_j(x)\le \Phi_{M}(x)$ for all
$j \le M$, $\Phi_i(x)<\Phi_{-N}(x)$ for all $i < -N$ and $\Phi_j(x)
\ge \Phi_{-N}(x)$ for all $j \ge -N$. Since $x$ and $y$ differ in
only a finite number of coordinates, for any $M,N$ large enough the
condition above will hold for $x$ and $y$ simultaneously. Thus, for
such $M,N$ let $w_1=x_{[-N,M]}$ and $w_2=y_{[-N,M]}$ then
$y=g_{w_1,w_2,-N}(x)$. This proves that the collection
$\mathfrak{C}_+$ generates $\T_{\hX_{+}}$.

Now we show that $\mu_{+}$ is invariant for any $g \in
\mathfrak{C}_+$. Let $w_1,w_2$ be admissible words as above. It
follows from direct computation that:

$$\mu_{+}([w]_k \cap A_{k,n})=\mu_{+}([w]_0)\mu_{+}(A_{k,n})$$
and
$$\mu_{+}([w]_0)=(\frac{1}{N})^{\Phi(w)}(p)^{\frac{n+\Phi(w)}{2}}(1-p)^{\frac{n-\Phi(w)}{2}}=$$
$$=\left(\sqrt{\frac{p}{1-p}}\frac{1}{N}\right)^{\Phi(w)}\left(p(1-p) \right)^{\frac{n}{2}}=$$
$$= \left( \frac{\sqrt{N^2}}{N}\right)^{\Phi(w)}\left(p(1-p) \right)^{n}.$$
Where $$A(w) =\{x \in \hX:~ \Phi_j(x)<0 \,\forall j<0 \mbox{ and }
\Phi_l(x)>\Phi(w) \forall l>n \}$$

Because this number is determined by $n$ and does not depend on the
value of $w$, this proves that $g$ preserves the measure of any such
set. It follows that $\mu_+$ is indeed $\T$-invariant.
\end{proof}

We will now identify an uncountable family of mutually singular
$\T_{\hX}$-invariant measures. This demonstrates a  dramatic failure
of the conclusion of Dobru{\v{s}}in's theorem for Kalikow-type
subshift: In absence of Dobru{\v{s}}in's condition, many ergodic
translation invariant Gibbs measures which are not equilibrium can
occur. Whenever $X=\{1,\ldots,N\}^\ZZ$, the subshift $\hX$ as above
has uncountably many ergodic $\T_{\hX}$-invariant measures which are
not probabilities, and the shift-invariant maps of $X$ can be mapped
via an injection into the $\T_{\hX}$-invariant maps of $\hX$.
Furthermore, applying the result of Kalikow, $\hX$ admits a measure
$\T_{\hX}$-invariant and shift invariant, is $K$ but not Bernoulli.
\begin{prop}
Let $\nu$ be a non-atomic $\sigma$-invariant probability measure on
$X$. Let $\widehat{\nu}=(P\times \nu)\circ\pi^{-1}$ where $P$ is the
symmetric Bernoulli measure on $\{+1,-1\}^\ZZ$, and
$\pi:\{-1,+1\}^\ZZ\times X \to \hX$ is the factor-map described in
the beginning of this section. The measure $\widehat{\nu}$ is
$\T_{\hX}$-invariant.
\end{prop}
\begin{proof}
To prove $\T_{\hX}$-invariance of $\widehat{\nu}$, we will describe
a set of $\T_{\hX}$-holonomies which generate $\T_{\hX}$ restricted
to a $\T_{\hX}$-saturated set of full $\widehat{\nu}$-measure, and
show that each of these holonomies preserves the measure
$\widehat{\nu}$.

Let $\hX_0$ be the set of points $x\in \hX$ such that $\liminf_{n
\to \pm \infty}\Phi_n(x)=-\infty$ and $\limsup_{n \to \pm
\infty}\Phi_n(x)=+\infty$. Evidently, $\hX_0$ is a Borel set, it is
saturated set with respect to the Gibbs relation of $\hX$ and
$\widehat{\nu}(\hX_0)=1$.

We already explained that the restriction of
$\pi:\{-1,+1\}^\ZZ\times X \to \hX$  to $\pi^{-1}\hX_0$ is
injective. This enables us to define a  Borel function $S:\hX_0 \to
X$ by $S:= \mathit{Scn}\circ \pi^{-1}$ where
$\mathit{Scn}:\{-1,1\}^\ZZ\times X \to X$ is the obvious projection
onto $X$.
 For $x \in \hX_0$ and $
n \in \ZZ$, $S(x)_n:=a$ iif $\mathit{Scn}\circ\sigma^k x = a$ for
some (hence all) $k \in \ZZ$ such that $\Phi_k(x)=n$. Denote by
$\hX_1$ the subset of $\hX_0$ which consists of all points $x \in
\hX_0$ such that $S(x)$ is not periodic:
$$\hX_1 =\{x \in \hX_0:~ \not\exists  k \in \ZZ\setminus\{0\} ~ S(x)=\sigma^kS(x) \} $$

Let $a,b \in \{+1,-1\}^n$ with $\Phi(a)=\Phi(b)$, and $k \in \ZZ$.
We define a $\T(\hX)$-holonomy $g_{a,b;k}:(\mathit{Wlk}^{-1}[a]_k
\cap \hX_0) \to (\mathit{Wkj}^{-1}[b]_k \cap \hX_0)$, where
$\mathit{Wlk}:\hX \to \{-1,+1\}^\ZZ$ is the obvious projection. The
function $g_{a,b;k}$ changes the walk in the time-interval $[k,k+n]$
for $a$ to $b$, making the required ``rearrangements'' of the
scenery using the function $S$, so that
$S(\sigma^k(x))=S(\sigma^k(y))$. Formally, $g_{a,b;k}$ is defined
by:
$$g_{a,b;k}(x)_m = \left\{ \begin{array}{cc}
(b_{m+k},S(\sigma^{k}x)_{\Phi(a_1^{m-k})}) & k\le m \le n\\  x_n &
\mbox{else}\end{array} \right.
$$
 Because $\Phi(a)=\Phi(b)$, $x$ and
$g_{a,b;k}$ are in same positions in the scenery, outside the time
interval $[k,k+n]$. In other words, the condition $\Phi(a)=\Phi(b)$
grantees that $g_{a,b;k}(x) \in \hX$, because the excursion length
by $x$ from $k$ to $k+n$ is the same as the that of $g_{a,b}(x)$.

We now verify that $g_{a,b;k}$ preserves the measure $\tilde{\nu}$:
 For any cylinder $[a]_k \subset \hX$,
$$\tilde{\nu}([a]_k)=2^{-\Phi(a)}\nu(S([a]_0))$$
The function $g_{a,b;k}$ was defined so that
$S(g_{a,b;k}([a]_0))=S([a]_0)$  and $\Phi(a)=\Phi(b)$. We now see
that $\tilde{\mu}([a]_k)=\tilde{\mu}(g_{a,b;k}([a]_k))$.

It remains to verify that the set of $\T(\hX)$-holonomies
$\{g_{a,b}\}$ as above generates $\T(\hX_1)$. Suppose $(x,y) \in
\T(\hX_1)$. Intuitively, this is because within the set $\hX_1$
every location in the scenery is visited infinitely often, and the
scenery is not periodic,  the only way to change a finite number of
coordinates in a consistent manner is by rearranging the walk in a
finite time interval, retaining the scenery and the offsets of the
endpoints of the walk in this time-interval.

Here is a formal proof of this: There exist $n \in \NN$ such that
$x_k=y_k$ for all $k \in \ZZ$ with $|k|>n$. Because $x,y \in \hX_0$,
it follows that $S(x)=S(y)$. since $S(x)=S(y)$ is not periodic a
periodic point (by definition of $X_1$), it also follows that
$\Phi(x_{[-n,n]})=\Phi(y_{[-n,n]})$. We conclude that $y=g_{a,b}(x)$
with $a=x_{[-n,n]}$ and $b=y_{[-n,n]}$.
\end{proof}

\section{$\T$-invariant measure of \label{sec:beta} $\beta$-shifts}
The subject of the following section is the identification of the
$\T$-invariant measure for a certain one-parameter family of
subshifts $X_\beta$ where $\beta >1$ is a real number.
%
%Recall the definition of $X_\beta$ as appearing in
%\cite{parry_beta}:
%and other references:
 Let $T_\beta$ be the self-map of the unit
interval $[0,1)$ given by $T_\beta(x)= \beta x \mod 1$.
Generalizing ordinary base $n$-expansions, the greedy $\beta$-expansion %(introduced in \cite{renyi})
 of $x \in [0,1)$ is the sequence $(a_1,a_2,\ldots)$ defined by $a_k = \lfloor
T_\beta^{k}x \rfloor$ $k \ge 1$. It satisfies the identity $x =
\sum_{k=1}^{\infty}a_k \beta^{-k}$.

The (two sided) \emph{$\beta$-shift} $X_\beta \subset \{1\ldots
\lfloor \beta \rfloor\}^\mathbb{Z}$ is the
 subshift whose admissible words are partial greedy $\beta$-expansions of numbers in $[0,1)$.

Denote the $\beta$-expansion of $1$ by
$\omega(\beta)=\omega_0\omega_1\ldots$, so that
$$1=\sum_{n=0}^{\infty}\beta^{-(n+1)}\omega_n.$$
Assume that $\omega$ does not terminate with $0$'s.

Here is a concrete description of $X_\beta$ (see
\cite{schmidt_beta}):
$$X_{\beta}= \left\{ x \in \{ 0,\ldots,\lfloor \beta \rfloor\}^\ZZ :~ x_{[n,n+m]} \preceq (\omega_0,\ldots,\omega_{m}) \forall n \in \ZZ,\, m\ge 0 \right\},$$

where $\preceq$ denotes the lexicographic order of words.

It follows easily from well known results (as in
\cite{schmeling_beta}) that for a residual set of $\beta$'s, the
topological Gibbs relation is trivial, and so theorem
\ref{thm:local_lanford_ruelle} above gives no information about
Gibbs-measures of $X_\beta$.

Nevertheless, it follows from a result of Walters
\cite{walters_beta} that $X_\beta$ has a unique equilibrium which is
also the unique Gibbs measure, for any $f:X_\beta \to \mathbb{R}$
which is H\"{o}lder continuous with respect to the metric
$d(x,y)=\exp(-\min\{ |n| :~ x_n \ne y_n\})$. In the spacial case
where $f$ is a constant function, this unique equilibrium, which is
the measure of maximal entropy projects onto a measure which is
absolutely continuous with respect to Lebesgue measure. Parry proved
this in \cite{parry_beta}, and gave the following formula for the
density function:
$$h_\beta(x)=\sum_{n=0}^{\infty}1_{[0,T^n1]}(x)\frac{1}{\beta^n}.$$

To complement this we prove the following:
\begin{thm}
  For any $\beta >1$, the $\beta$-shift $X_\beta$ has a unique
  $\T$-invariant measure.
\end{thm}

\begin{proof}

%First, some notations:
For an integer $k$, let $L_\beta(k)$ denote the set of $k$-tuples
which can appear admissibly in a $\beta$-expansion, e.i the cylinder
$[y]_0 \subset X_\beta$ is nonempty. Let $F_\beta(k)$ denote the set
of $k$-tuples which can appear admissibly in $X_\beta$, and can be
followed admissibly by any admissible sequence. Also let
$$F_\beta = \left\{ (\ldots,x_{-n},\ldots,x,_{-2},x_{-1}) \in
\{1,\ldots,\lfloor \beta \rfloor\}^{-\NN}:~ \forall n<0~ x_n^0 \in
F_\beta(n)\right\}.$$ That is, $F_\beta$ is the set of left-infinite
sequences which can be followed admissibly by and admissible
sequence in $X_\beta$.

Assume $\mu \in \mathcal{P}(x_\beta)$ be  $\T$-invariant. Our
strategy is to show that $\mu$ satisfies various properties, which
eventually determine $\mu$ uniquely.

  Fix $n \in \mathbb{N}$. For $N>n$, define $p_N= \mu(\{x\in X_\beta :\;
x_{-n}^{-N}=\omega_{1}^{n-N+1}\})$. For any $y \in L_{\beta}(k)$
there exists a tail holonomy $[\omega_1^{k-1}]_0 \rightarrow [y]_0$,
so $\mu([\omega_1^{k-1}]_0]) \le \mu([y]_0)$. It follows that
$\mu([\omega_1^k]) \le \frac{1}{|L_\beta(k)|} = (\beta+o(1))^{-k}$,
so $\sum_{N} p_N $ is dominated by $\sum_k (\beta-\epsilon)^{-k} <
\infty$. By the Borel-Cantelli lemma we deduce that $\sum_{N} p_N <
\infty$, and so:
\begin{equation}\label{eq:beta_w_tail_null}
  \mu(\bigcap_{N>k}\bigcup_{n>N}\{x\in X_\beta :\; x_{-n}^{-N}=\omega_{1}^{n-N+1}\})=0
\end{equation}
Now let  $B_N= \bigcap_{n >N}\{x \in X_\beta :\; x_{\infty}^{-n}
\not\in F_\beta \}$.

Let $x \in B_N$. For any $n>N$ there exists $k\ge n$ so that
$x_{-k}^{-n}=w_1^{k-n+1}$. Thus, there exists a sequence of integers
$N \le n_1 < n_2 < \ldots$ so that
$x_{-n_{i+1}}^{-n_i}=\omega_1^{n_{i+1}-n_i+1}$.  We claim that for
any $i>0$, $x_{-n_i}^{N}=\omega_1^{n_i-N+1}$. This follows by
induction as follows:

 Suppose $x_{-n_i}^{N}=\omega_1^{n_i-N+1}$.
Since $x \in X_\beta$ we have:
$$\omega_1^{n_{i+1}-N+1}\succeq x_{-n_{i+1}}^{n}$$
On the other hand, since $\omega_n^\infty \preceq \omega_1^{\infty}$
for $n \ge 1$ :
$$x_{-n_{i+1}}^{n}=\omega_{1}^{n_{i+1}-n_{i}+1}\omega_1^{n_i-N+1}  \succeq \omega_1^{n_{i+1}-N+1}$$
this proves the claim.

It follows that $$B_n \subset \bigcap_{k > N}\bigcup_{n
>k}\{x_{-n}^{-N}= \omega_1^{n-N+1}\}.$$ By equation
\eqref{eq:beta_w_tail_null}, this implies that $\mu(B_N)=0$ for any
$N >0$.  Thus,
\begin{equation}\label{eq:beta_full_tail_sure}
\mu( \bigcap_{N>0}\bigcup_{n>N}\{ x \in X_\beta :x_{-\infty}^{-n}
\in F_\beta\})=1
\end{equation}

For $n
> |y|$, let $A_n=\{x \in X_\beta :\;
x_{-\infty}^{-n} \in F_\beta\}$. Let:
$$ f_n: = |F_\beta(n)|,~ b_n = |L_\beta(n)|$$

Fix $y \in F_\beta(k)$. Define:  $$F_n(y):= \{ a \in F_\beta(n):\;
a_{n-k+1}^{k}=y\},~ B_n(y) = \{ a \in L_\beta(n):\;
a_{n-k+1}^{k}=y\},$$ $$ f_n(y) = | F_N(y)|,~  b_n(y) = |B_n(y)|.$$
We now prove the following inequalities:
\begin{equation}\label{qe:beta_cylinder_mu_bounds}
 \frac{f_n(y)}{f_n} \le \mu( [y]_{-k} | A_n) \le \frac{b_n(y)}{b_n}
\end{equation}

 For any $y_1,y_2 \in F_\beta(n)$, there is a natural
$\mathcal{T}(X_\beta)$-holonomy $[y_1]_{-n}\cap A_n \leftrightarrow
[y_2]_{-n} \cap A_n$ defined by changing the $n$ coordinates in the
interval $[-n,-1]$ from $y_1$ to $y_2$. This function is well
defined, because if $x_- \in F_\beta$ and $a \in F_\beta(n)$ then
$x_-a \in F_\beta$. It follows that
$\mu([y_1]_{-n}|A_n)=\mu([y_2]_{-k}|A_n)$ for any $y_1,y_2 \in
F_\beta(n)$. Similarly, if $y_1 \in L_\beta(n)$ and $y_2 \in
F_\beta(n)$ then $\mu([y_1]_{-n}|A_n) \le \mu([y_2]_{-k}|A_n)$.
 Since $\biguplus_{a \in D_n}[a]_{-n} \subset [y]_{-k} \subset \biguplus_{ a \in
 E_n}[a]_{-n}$,  the inequality \eqref{qe:beta_cylinder_mu_bounds} follows.

 Let us estimate $f_n$,$b_n$,$f_n(y)$ and $b_n(y)$. The
 following recursive formulas hold for $n \ge k$:
 $$f_n(y)= \sum_{i=1}^{n-|a|}\omega_i d_{n-i} + 1_{[\omega_{n+1}^{n+k} \prec
y]}$$
$$b_n(y)= \sum_{i=1}^{n-|a|}\omega_i e_{n-i} + 1_{[\omega_{n+1}^{n+k}
\preceq y]}$$ with $f_k=e_k = 1$. Define sequences $u_n$,$v_n$ by:

$$u_n=\sum_{i=1}^{n-1}\omega_i u_{n-i} + 1 \mbox{ and }
 v_n=\sum_{i=1}^{n-1}\omega_i v_{n-i} $$ with initial condition
$u_k=v_k=1$. We have the following inequalities: $v_n \le f_n(y) \le
b_n(y) \le v_n$. Denote $x_n= \frac{v_n}{\beta^n}$. From the
recursive formula of $v_n$ it follows that:
$$x_n=\sum_{k=1}^{n-1}\frac{\omega_k}{\beta^k}x_{n-k},$$
and the initial condition $x_k =\frac{1}{\beta^k}$.

 Since $0<1- \sum_{k=1}^{n}\frac{\omega_k}{\beta^k} < \frac{1}{\beta^n}$,
 $x_{n+1} \ge \min_{1 \le i \le n}x_i
(1-\beta^{-n})$, So by induction: $x_{n} \ge
\frac{1}{\beta^k}\prod_{k=1}^{n}(1-\beta^{-k})$. We conclude that
$c= \inf_n x_n \ge x_1 \prod_{j=1}^{\infty}(1-\beta^{-j})>0$, so for
this $c$, $c\beta^{n-k} \le v_n \le \beta^{n-k} $

By induction, we show that $u_n= \sum_{j=1}^n v_n$. So $v_n \le
\beta^{-k} \sum_{k=1}^n\beta^{-j} \le
\beta^{-k}\frac{\beta}{\beta-1}$, and the ratios $\frac{u_n}{v_n}$
are bounded as follows:
$$1 \le \frac{u_n}{v_n} \le  c^{-1}\frac{\beta}{\beta-1}$$

The following recursions also hold:
$$ f_{n+1}= \sum_{j=1}^{n}\omega_j f_{n-j} \mbox{ and } b_{n+1}= \sum_{j=1}^{n}\omega_j b_{n-j} +1$$
with initial conditions $f_1=\omega_1$ and $b_1 = \omega_1 +1$.

 Repeating the above
calculations, we see that $c \beta^{n} \le f_{n} \le \beta^{n}$ and
$$1 \le \frac{b_n}{f_n} \le  c^{-1}\frac{\beta}{\beta-1}$$
it follows that
$$\frac{f_n(y)}{f_n} \ge
\beta^{-k}c^{-1}\frac{\beta}{\beta-1}$$ and
$$ \frac{b_n(y)}{b_n} \le c\beta^{-k}\frac{\beta-1}{\beta}$$

So by the inequalities \eqref{qe:beta_cylinder_mu_bounds}, it
follows that for some $c>0$ depending on  $\beta$,  $\mu([y]_{-k}|
A_n)= c^{\pm 1}\beta^{-k}$. By equation
\eqref{eq:beta_full_tail_sure}, $\mu(\bigcup_{n} A_n) =1$. It thus
follows that for some positive $c$,
\begin{equation}\label{eq:beta_abs_cont_full_cylinders}
   \mu([y]_{-k})= c^{\pm 1}\beta^{-k}.
\end{equation}

From equation \eqref{eq:beta_full_tail_sure}, it follows that any
cylinder (thus any Borel set), can be approximated by a union of
cylinders of the form $[y]_k$ with $y \in F_\beta(n)$ for some $n$.
Concluding, we have shown that any two $\T$-invariant probability
measures on $X_\beta$ are absolutely continuous. This implies that
there is at most one such measure - which we know must be the unique
equilibrium.
\end{proof}

\section{\label{sec:dyck} A Lanford-Ruelle theorem for The Dyck-Shift}
The Dyck-shift is a certain subshift whose origin is in the study of
formal languages. The Dyck shift is an interesting example to study
when attempting to extend a theory for which applies to SFTs, as in
\cite{hamachi_inoue_dyck,krieger_74,krieger_matsumoto_2003,meyerovitch_dyck}.
For completeness, we write the definition of the Dyck-shift $D=D_2
\subset S^\mathbb{Z}$:%, which can be found in any of the references
%above:

Let $S=\{\alpha_1,\alpha_2,\alpha_1^{-1},\alpha_2^{-1}\}$. Define an
associative operation $$\ast: S^*\cup \{0\} \times S^* \cup \{0\}
\to S^* \cup \{0\}$$ via concatenation subject to the following
reduction rules:\\
\begin{enumerate}
\label{def_m}
  \item \(\alpha_j \ast \alpha_j^{-1} = \Lambda\).
  \item \(\alpha_i \ast \alpha_j^{-1} = 0  \) for $i \ne j$.
\end{enumerate}

where $\Lambda$ denotes the empty word.

$D$ is defined by forbidding any word which reduces to $0$ by
application of the reduction rules above. A suggestive
interpretation is to regard the $\alpha_i$'s as two type of
``open-brackets'' and their inverses as closing brackets. Think of
$\alpha_1$ is a ``round open bracket'', of $\alpha_2$ as a ``square
open bracket'' and of the $\alpha_i^{-1}$'s as closed brackets of
the corresponding types.

There are $2$ ergodic measures of maximal entropy for $D$, as shown
in \cite{krieger_74}. In \cite{meyerovitch_dyck} all $\T$-invariant
measures for the Dyck-shift were explicitly described. This set of
$\T$-invariant measures consists the convex-hull of the two measures
of maximal entropy plus a third ergodic shift-invariant measure.

A function $f:D \to \mathbb{R}$ is a \emph{site-potential} if it is
of the form
$$f(x) = \begin{cases}
  f_1 & x_0 = \alpha_1\\
  f_2 & x_0 = \alpha_2\\
  f_3 & x_0 = \alpha_1^{-1}\\
  f_4 & x_0 = \alpha_2^{-1}\\
\end{cases}$$

where $f_1,\ldots,f_4 \in \RR$.

We now show that the following restricted Lanford-Ruelle-type
theorem holds for the Dyck-Shift:

\begin{prop}\label{prop:dyck_lanford_ruelle}
If $f:D \to \mathbb{R}$ is a site-potential, then any
$f$-equilibrium is $f$-Gibbs.
\end{prop}
\begin{proof}
Let $\mu\in \mathcal{P}(D,T)$ be an ergodic shift-invariant measure,
and denote $\mu_i =\mu(f^{-1}(f_i))$, $\mu+=\mu_1 + \mu_2$ and
$\mu_-=\mu_3 + \mu_4$.

With out loss of generality, suppose $\mu_+ \ge \mu_-$. By
ergodicity of $\mu$ this means that almost surely the frequency of
open brackets is at least equal to that of closed brackets.
 It follows that almost-surely there are no unmatched closed brackets, thus $\mu_1 \ge \mu_3$ and
$\mu_2 \ge \mu_4$. Define two more parameters of the measure $\mu$:
$\mu_1^+$ is the probability that coordinate $0$ of $x$ is a square
bracket, given that it is an open bracket which is unmatched.
$\mu_2^+=1-\mu_1^+$ is the probability that coordinate $0$ of $x$ is
a round bracket, given that it is an open bracket which is
unmatched.

The integral of $f$ with respect to $\mu$ can be expressed by the
parameters:
$$\mu(f)= \int fd\mu = \sum_{i=1}^4f_i\mu_i$$
Using Roklin's formula for relative entropy we have:
$$h(\mu) \le
H(\mu_+,\mu_-)+(\mu_+-\mu_-)H(\mu_1^+,\mu_2^+)+\mu_-H(\frac{\mu_3}{\mu_3+\mu_4},\frac{\mu_4}{\mu_3+\mu_4})$$
$$=H(\mu_3,\mu_4,\mu_+)+(\mu_+-\mu_-)H(\mu_1^+,\mu_2^+)$$

Equality in the above holds  iff the partition of $D$ according to
the direction (open /close) of the first bracket is i.i.d and the
types of brackets are independent jointly independent accept for the
obvious restriction that matching brackets are of the same type.

Observe that the probability of having an unmatched open bracket at
coordinate $0$ is $\mu_+ - \mu_-$. By ergodicity of $\mu$,
almost-surely the frequency of matched opening round brackets is
equal to the frequency of closing round brackets, so the probability
of a matched opening  round bracket is $\mu_1$. The discussion above
yields the following relation between our parameters:
$$\mu_1=\mu_1^+(\mu_+-\mu_-)+\mu_3$$
By repeating a similar argument for square brackets, or using the
linear relations between the parameters, we also obtain:
$$\mu_2=\mu_2^+(\mu_+-\mu_-)+\mu_4$$
%substituting variables in the above relation:
%$$\mu_1=\mu_1^+(1-2(\mu_3+\mu_4))+\mu_3$$
%$$\mu_2=(1-\mu_1^+)(1-2(\mu_3+\mu_4))+\mu_4$$

 It follows that the
$f$-pressure of $\mu$ satisfies $P_\mu(f) \le
P(\mu_1^+,\mu_3,\mu_4)$, where:
$$P(\mu_1^+,\mu_3,\mu_4) = H(\mu_+,\mu_3,\mu_4)+
(\mu_+-\mu_-)H(\mu_1^+,\mu_2^+)+\sum_{i=1}^4f_i\mu_i$$
%
%$$P(\mu_1^+,\mu_3,\mu_4) = H((1-\mu_3-\mu_4),\mu_3,\mu_4)+
%(1-2(\mu_3+\mu_4))H(\mu_1^+,1-\mu_1^+)+$$
%$$+(f_1+f_3)\mu_3+(f_2+f_4)\mu_4+f_1(1-2(\mu_3+\mu_4))\mu_1^++f_2(1-2(\mu_3+\mu_4))\mu_1^+$$

There is a unique shift-invariant measure $\mu \in \mathcal{P}(D)$
with specified parameters $\mu_1,\ldots,\mu_4,\mu_1^+,\mu_2^+$ for
which the above inequality is an equality. To describe this measure,
it is will be convenient to describe the stochastic process
$x=(x_n)_{n \in \ZZ}$ such that $x \in D$:  Let $$y_n =
\begin{cases}
  a & \mbox{if } x_n=\alpha_1 \mbox{ or } x_n=\alpha_2\\
  b & \mbox{if } x_n=\alpha_1^{-1} \mbox{ or } x_n=\alpha_2^{-1}\\
\end{cases}$$
Under the distribution of $\mu$, the process $(y_n)_{n \in \ZZ}$ is
i.i.d with marginal distribution $(\mu_+,\mu_-)$. Given a
realization of factor process $y$, whenever $y_n=b$,
$$\mu(x_n =\alpha_1^{-1})=\mu_3/\mu_-,$$
 and
 $$\mu(x_n= \alpha_2^{-1})=\mu_4/\mu_-,$$
 independently of the other
coordinates. Given that $y_n=b$ and $x_n$ does not have a matching
closing bracket, then the probability that $x_n=\alpha_1$ is
$\mu_1^+$. This completes the description of $\mu$, depending on the
parameters, since if $y_n=b$ and the $\alpha_i$ at $x_n$ has a
matching $\alpha_i^{-1}$, then $x_n$ is determined according to the
type of the closing bracket. It follows that an $f$-equilibrium
measure will be of the form $\mu$ above, for which the parameters
$(\mu_1,\ldots,\mu_4,\mu_1^+,\mu_2^+)$ maximize the expression
$P(\mu_1,\ldots,\mu_4,\mu_1^+,\mu_2^+)$, under the obvious linear
restrictions on these parameters. It is now an elementary problem to
maximize this expression. Using the restrictions on parameters, we
can reduce to three parameters $\mu_3,\mu_4,\mu_1^+$, which are
independent and need only satisfy $\mu_3 \ge 0$, $\mu_4 \ge 0$,
$\mu_3+\mu_4 \le \frac{1}{2}$ and $0 \le \mu_1^+ \le 1$.

The following expressions are the partial derivatives of $P$:
$$\frac{\partial P}{\partial \mu_3} =  \log\left(
\frac{1-\mu_3-\mu_4}{\mu_3}\right)-2H(\mu_1^+,1-\mu_1^+)+$$
$$+(f_1+f_3)-2f_1\mu_1^+-2f_2(1-\mu_1^+)$$
$$\frac{\partial P}{\partial \mu_4} =  \log\left(
\frac{1-\mu_3-\mu_4}{\mu_4}\right)-2H(\mu_1^+,1-\mu_1^+)+$$
$$+(f_2+f_4)-2f_1\mu_1^+-2f_2(1-\mu_1^+)$$
 and
$$\frac{\partial P}{\partial \mu_1^+} = (1-2(\mu_3 +\mu_4))\log\left(
\frac{1-\mu_1^+}{\mu_1^+}\right)+(f_1-f_2)(1-2(\mu_3 +\mu_4))$$

The equation $\frac{\partial P}{\partial \mu_3} - \frac{\partial
P}{\partial \mu_4}=0$ gives
\begin{equation}\label{eq:dyck_gibbs1}
  \frac{\mu_3}{\mu_4}= \exp(f_1+f_3 - (f_2+f_4)).
\end{equation}

The equation $\frac{\partial P}{\partial \mu_1^+} =0$ gives
\begin{equation}\label{eq:dyck_gibbs2}
\frac{\mu_1^+}{\mu_2^+} = \exp( f_1 -f_2)
\end{equation}

Substituting $\log(\mu_1^+) = \log(\mu_2^+) +f_1 -f_2$ into the
equation $\frac{\partial P}{\partial \mu_3} =0$, we get:
$$0=\log \mu_+ -\log \mu_3
+\mu_1^+(f_1-f_2+\log\mu_2^+)+2\mu_2^+\log \mu_2^+ +f_1 +f_3 -
2f_1\mu_1^+ -2f_2\mu_2^+.$$ Simplifying and rearranging the above
equation, we get:
\begin{equation}\label{eq:dyck_gibbs3}
  \frac{(\mu_+\mu_2^+)^2}{\mu_+\mu_3} = \exp(f_1 +f_3 -2f_2)
\end{equation}

Equations \eqref{eq:dyck_gibbs1},\eqref{eq:dyck_gibbs2} and
\eqref{eq:dyck_gibbs3} are necessary and sufficient conditions on
the parameters $\mu_3,\mu_4,\mu_1^+$ for $\mu$ to be $(\log
f,\T)$-conformal.

\end{proof}

\bibliographystyle{abbrv}
\bibliography{gibbs_subshifts}

\begin{thebibliography}{10}

\bibitem{aaro-nakada-2007}
J.~Aaronson and H.~Nakada.
\newblock Exchangeable, gibbs and equilibrium measures for markov subshifts.
\newblock {\em Ergodic Theory Dynam. Systems}, 27(2):321--339, 2007.

\bibitem{bowen_equlibrium}
R.~Bowen.
\newblock Some systems with unique equilibrium states.
\newblock {\em Math. Systems Theory}, 8(3):193--202, 1974/75.

\bibitem{burton_steif94}
R.~Burton and J.~E. Steif.
\newblock Non-uniqueness of measures of maximal entropy for subshifts of finite
  type.
\newblock {\em Ergodic Theory Dynam. Systems}, 14(2):213--235, 1994.

\bibitem{dobrusin_gibbs}
R.~L. Dobru{\v{s}}in.
\newblock Gibbsian random fields for lattice systems with pairwise
  interactions.
\newblock {\em Funkcional. Anal. i Prilo\v zen.}, 2(4):31--43, 1968.

\bibitem{feldman_moore_77_i}
J.~Feldman and C.~C. Moore.
\newblock Ergodic equivalence relations, cohomology, and von {N}eumann
  algebras. {I}.
\newblock {\em Trans. Amer. Math. Soc.}, 234(2):289--324, 1977.

\bibitem{hamachi_inoue_dyck}
T.~Hamachi and K.~Inoue.
\newblock Embedding of shifts of finite type into the {D}yck shift.
\newblock {\em Monatsh. Math.}, 145(2):107--129, 2005.

\bibitem{kalikow82}
S.~A. Kalikow.
\newblock {$T,\,T\sp{-1}$} transformation is not loosely {B}ernoulli.
\newblock {\em Ann. of Math. (2)}, 115(2):393--409, 1982.

\bibitem{krieger_74}
W.~Krieger.
\newblock On the uniqueness of the equilibrium state.
\newblock {\em Math. Systems Theory}, 8(2):97--104, 1974/75.

\bibitem{kreiger_dimension_1980}
W.~Krieger.
\newblock On dimension functions and topological {M}arkov chains.
\newblock {\em Invent. Math.}, 56(3):239--250, 1980.

\bibitem{krieger_matsumoto_2003}
W.~Krieger and K.~Matsumoto.
\newblock A lambda-graph system for the {D}yck shift and its {$K$}-groups.
\newblock {\em Doc. Math.}, 8:79--96 (electronic), 2003.

\bibitem{lanford_robinson_68}
O.~E. Lanford, III. and D.~W. Robinson.
\newblock Statistical mechanics of quantum spin systems. {III}.
\newblock {\em Comm. Math. Phys.}, 9:327--338, 1968.

\bibitem{lanford_ruelle}
O.~E. Lanford, III and D.~Ruelle.
\newblock Observables at infinity and states with short range correlations in
  statistical mechanics.
\newblock {\em Comm. Math. Phys.}, 13:194--215, 1969.

\bibitem{marcus_newhouse_78}
B.~Marcus and S.~Newhouse.
\newblock Measures of maximal entropy for a class of skew products.
\newblock In {\em Ergodic theory (Proc. Conf., Math. Forschungsinst.,
  Oberwolfach, 1978)}, volume 729 of {\em Lecture Notes in Math.}, pages
  105--125. Springer, Berlin, 1979.

\bibitem{meyerovitch_dyck}
T.~Meyerovitch.
\newblock Tail invariant measures of the {D}yck shift.
\newblock {\em Israel J. Math.}, 163:61--83, 2008.

\bibitem{parry_beta}
W.~Parry.
\newblock On the {$\beta $}-expansions of real numbers.
\newblock {\em Acta Math. Acad. Sci. Hungar.}, 11:401--416, 1960.

\bibitem{petersen_schmidt97}
K.~Petersen and K.~Schmidt.
\newblock Symmetric {G}ibbs measures.
\newblock {\em Trans. Amer. Math. Soc.}, 349(7):2775--2811, 1997.

\bibitem{ruelle_themo}
D.~Ruelle.
\newblock {\em Thermodynamic formalism}.
\newblock Cambridge Mathematical Library. Cambridge University Press,
  Cambridge, second edition, 2004.
\newblock The mathematical structures of equilibrium statistical mechanics.

\bibitem{schmeling_beta}
J.~Schmeling.
\newblock Symbolic dynamics for {$\beta$}-shifts and self-normal numbers.
\newblock {\em Ergodic Theory Dynam. Systems}, 17(3):675--694, 1997.

\bibitem{schmidt_super_k}
K.~Schmidt.
\newblock Invariant cocycles, random tilings and the super-{$K$} and strong
  {M}arkov properties.
\newblock {\em Trans. Amer. Math. Soc.}, 349(7):2813--2825, 1997.

\bibitem{schmidt_beta}
K.~Schmidt.
\newblock Algebraic coding of expansive group automorphisms and two-sided
  beta-shifts.
\newblock {\em Monatsh. Math.}, 129(1):37--61, 2000.

\bibitem{walters_g_func}
P.~Walters.
\newblock Ruelle's operator theorem and {$g$}-measures.
\newblock {\em Trans. Amer. Math. Soc.}, 214:375--387, 1975.

\bibitem{walters_var}
P.~Walters.
\newblock A variational principle for the pressure of continuous
  transformations.
\newblock {\em Amer. J. Math.}, 97(4):937--971, 1975.

\bibitem{walters_beta}
P.~Walters.
\newblock Equilibrium states for {$\beta $}-transformations and related
  transformations.
\newblock {\em Math. Z.}, 159(1):65--88, 1978.

\end{thebibliography}
\end{document}